\documentclass{article}[12pt]
\usepackage{amsfonts,amssymb}

\usepackage{a4wide}
\usepackage{graphicx}
\usepackage{amsthm}
\usepackage{amsmath}
\usepackage{mathrsfs}

\newtheorem{lemma}{Lemma}[section]
\newtheorem{theorem}{Theorem}[section]

\begin{document}

\markboth{Sacha C. Blumen}
{Uncoloured $U_{q}(osp(1|2n))$ and $U_{-q}(so({2n+1}))$ link invariants}

\title{
ON THE $U_{q}(osp(1|2n))$ AND $U_{-q}(so({2n+1}))$ UNCOLOURED QUANTUM LINK INVARIANTS
\footnote{Running title: UNCOLOURED $U_{q}(osp(1|2n))$ AND $U_{-q}(so({2n+1}))$ 
LINK INVARIANTS}}

\author{SACHA C. BLUMEN 
\\
School of Mathematics and Statistics,  
\\
The University of Sydney, NSW, 2006, Australia.}

\maketitle

\begin{abstract}

Let $L$ be a link and 
$\Phi^{A}_{L}(q)$ its link invariant associated with the vector representation of the 
quantum (super)algebra $U_{q}(A)$.
Let $F_{L}(r,s)$ be the Kauffman link invariant for $L$ associated with the
Birman--Wenzl--Murakami algebra $BWM_{f}(r,s)$ for complex parameters $r$ and $s$ and a 
sufficiently large rank $f$.

For an arbitrary link $L$, we show that
$\Phi^{osp(1|2n)}_{L}(q) = F_{L}(-q^{2n},q)$ and 
$\Phi^{so({2n+1})}_{L}(-q) = F_{L}(q^{2n},-q)$ for each positive integer $n$ 
and all sufficiently large $f$, and that 
$\Phi^{osp(1|2n)}_{L}(q)$ and $\Phi^{so({2n+1})}_{L}(-q)$ are identical up to a substitution
of variables.

For at least one class of links $F_{L}(-r,-s) = F_{L}(r,s)$ implying
$\Phi^{osp(1|2n)}_{L}(q) = \Phi^{so({2n+1})}_{L}(-q)$ for these links.

\end{abstract}



\section{Introduction}	

Let $L$ be a link and 
$\Phi^{A}_{L}(q)$ the link invariant for $L$ associated with the vector representation of the 
quantum (super)algebra $U_{q}(A)$.
Let $F_{L}(r,s)$ be the Kauffman link invariant for $L$ associated with 
$BWM_{f}(r,s)$, the Birman--Wenzl--Murakami algebra of sufficiently large rank $f$ 
and complex parameters $r$ and $s$.
We are here interested in the invariants $\Phi^{osp(1|2n)}_{L}(q)$ and
$\Phi^{so({2n+1})}_{L}(-q)$ and will prove the following theorems.

\begin{theorem}
\label{th:quantuminvariantequaltokauffman}
For an arbitrary link $L$ and each positive integer $n$,
\begin{itemize}
\item[(i)] $\Phi^{osp(1|2n)}_{L}(q) = F_{L}(-q^{2n},q)$, and 
\item[(ii)] $\Phi^{so({2n+1})}_{L}(-q) = F_{L}(q^{2n},-q)$.
\end{itemize}
\end{theorem}

Recall that the braid group on $l$ strings, $B_{l}$, has generators  
$\{\sigma_{1}, \sigma_{2}, \ldots, \sigma_{l-1}\}$ satisfying the relations
\begin{eqnarray}
\sigma_{i}\sigma_{j} & = & \sigma_{j}\sigma_{i}, \hspace{15mm} |i-j| > 1,  \label{braidrel1} \\
\sigma_{i} \sigma_{i+1} \sigma_{i} & = & \sigma_{i+1} \sigma_{i} \sigma_{i+1}. \label{braidrel2}
\end{eqnarray}
Figure \ref{fig:generatorsigma} shows a graphical representation of $\sigma_{i}$ and $\sigma_{i}^{-1}$.

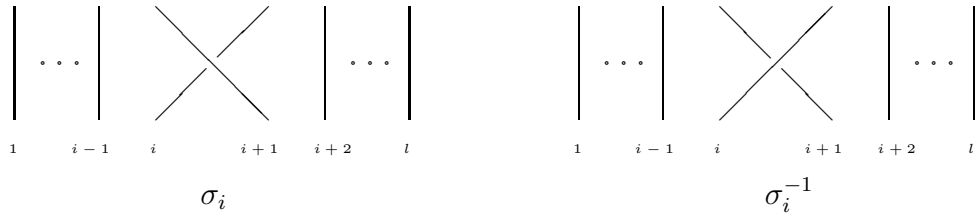
\begin{figure}[hbt]
\label{fig:generatorsigma}
\begin{center}
\setlength{\unitlength}{0.15mm}
\begin{picture}(850,200) \large\sf
\thinlines

\put(125,175){\line(1,-1){100}} 
\put(200,45){\tiny{\text{$i+1$}}}

\put(180,130){\line(1,1){45}}
\put(125,75){\line(1,1){45}}
\put(120,45){\tiny{\text{$i$}}}

\put(0,75){\line(0,1){100}}
\put(-5,45){\tiny{\text{$1$}}}
\put(75,75){\line(0,1){100}}
\put(50,45){\tiny{\text{$i-1$}}}
\put(275,75){\line(0,1){100}}
\put(265,45){\tiny{\text{$i+2$}}}
\put(350,75){\line(0,1){100}}
\put(345,45){\tiny{\text{$l$}}}

\put(25,125){\circle{2}} 
\put(40,125){\circle{2}} 
\put(55,125){\circle{2}} 

\put(300,125){\circle{2}} 
\put(315,125){\circle{2}} 
\put(330,125){\circle{2}} 

\put(165,0){\text{$\sigma_{i}$}}

\put(625,75){\line(1,1){100}} 
\put(620,45){\tiny{\text{$i$}}}

\put(625,175){\line(1,-1){45}}
\put(680,120){\line(1,-1){45}}
\put(700,45){\tiny{\text{$i+1$}}}

\put(500,75){\line(0,1){100}}
\put(495,45){\tiny{\text{$1$}}}
\put(575,75){\line(0,1){100}}
\put(550,45){\tiny{\text{$i-1$}}}
\put(775,75){\line(0,1){100}}
\put(765,45){\tiny{\text{$i+2$}}}
\put(850,75){\line(0,1){100}}
\put(845,45){\tiny{\text{$l$}}}

\put(525,125){\circle{2}} 
\put(540,125){\circle{2}} 
\put(555,125){\circle{2}} 

\put(800,125){\circle{2}} 
\put(815,125){\circle{2}} 
\put(830,125){\circle{2}} 

\put(665,0){\text{$\sigma^{-1}_{i}$}}


\end{picture}
\caption{The generators $\sigma_{i}$ and $\sigma^{-1}_{i}$ of $B_{l}$}  
\end{center}
\end{figure}

The following theorem is a key result.
\begin{theorem}
\label{th:equalityofkauffman}
Let $L(m)$, $m \in \mathbb{Z}$, be a link presented as the canonical closure of a braid
with corresponding braid group element $(\sigma_{1})^{m}$.
Then $F_{L(m)}(-r,-s) = F_{L(m)}(r,s)$ and $\Phi^{osp(1|2n)}_{L(m)}(q)=\Phi^{so({2n+1})}_{L(m)}(-q)$.
\end{theorem}

We have not proved a theorem corresponding to Theorem \ref{th:equalityofkauffman} for arbitrary links, but
have the following weaker result.
\begin{theorem}
\label{thm:equalsubstitutionofvariables}
For each arbitrary link $L$ and each positive integer $n$,
 $\Phi^{osp(1|2n)}_{L}(q)$ and $\Phi^{so({2n+1})}_{L}(-q)$ are equal up a substitution of variables. 
\end{theorem}

The background for the connection between $\Phi^{osp(1|2n)}_{L}(q)$ and ${\Phi^{so({2n+1})}_{L}(-q)}$ 
starts with Zhang \cite{zsuper}, who showed that an 
isomorphism exists between $U_{-q}(so({2n+1}))$ and $U_{q}(osp(1|2n))$ at generic $q$ (i.e. for $q$ not a root of unity)
where $U_{-q}(so({2n+1}))$ is restricted to finite dimensional tensorial highest weight irreducible representations and 
$U_{q}(osp(1|2n))$ is restricted to finite dimensional highest weight irreducible representations.

The Clebsch-Gordan coefficients for tensor products of these $U_{q}(osp(1|2n))$ irreps are identical
to those of
tensor products of the finite dimensional $U_{-q}(so({2n+1}))$ irreps with the same highest weights \cite{zsuper}.

Let $V$ be the module for the $({2n+1})$-dimensional irreducible (vector) representation $\pi_{V}$
of $U_{q}(so({2n+1}))$.
Then there also exists a $({2n+1})$-dimensional irreducible representation of $U_{q}(osp(1|2n))$ the module of which
we also denote by $V$.
Representations of $B_{l}$ exist in the $U_{q}(so({2n+1}))$ and $U_{q}(osp(1|2n))$ centralisers of $V^{\otimes g}$ for all $g \geq l$:
\begin{eqnarray*}
\rho^{so}: B_{l} & \rightarrow & \mathrm{End}_{U_{q}(so({2n+1}))}(V^{\otimes g}), \\
\rho^{osp}: B_{l} & \rightarrow & \mathrm{End}_{U_{q}(osp(1|2n))}(V^{\otimes g}),
\end{eqnarray*} 
and Markov traces can be defined on $\rho^{so}$ and $\rho^{osp}$ \cite{z1}.
Link invariants $\Phi^{so({2n+1})}_{L}(q)$ and 
$\Phi^{osp(1|2n)}_{L}(q)$ can then be defined from these Markov traces \cite{z1}.

The relationship between $\Phi^{osp(1|2n)}_{L}(q)$ and $\Phi^{so({2n+1})}_{L}(-q)$, for an arbitrary link
$L$, is unclear notwithstanding the limited isomorphism between $U_{-q}(so({2n+1}))$ and $U_{q}(osp(1|2n))$.
We will prove that the Bratteli diagrams for
certain semisimple quotients of $BWM_{f}(-q^{2n},q)$ and $BWM_{f}(q^{2n},-q)$ are identical for each fixed $f$.
An abstract symmetry between $\Phi^{osp(1|2n)}_{L}(q)$ and $\Phi^{so({2n+1})}_{L}(-q)$ 
is then implied from the combination of the fact that  
$F_{L}(-q^{2n},q)$ and $F_{L}(q^{2n},-q)$ are specialisations of $F_{L}(r,s)$
and Theorem \ref{th:quantuminvariantequaltokauffman}.

The reader should note that we would have $\Phi^{osp(1|2n)}_{L}(q) = \Phi^{so({2n+1})}_{L}(-q)$ for all links $L$
from Theorem \ref{th:quantuminvariantequaltokauffman} if we could extend the result in Theorem \ref{th:equalityofkauffman} 
to arbitrary links. We fix the notation $\mathbb{Z}_{+} = \{0, 1, 2, \ldots \}$.

\section{$U_{q}(osp(1|2n))$ and $U_{-q}(so({2n+1}))$ link invariants}

\subsection{Using Markov traces to define link invariants}
\label{subsect:linkinvarmarkovtrace}

Let $L$ be a link presented as the canonical closure of a braid on $f$ strings
that has the corresponding braid group element 
$b = \sigma_{i_{1}}^{m_{1}} \sigma_{i_{2}}^{m_{2}} \cdots \sigma_{i_{j}}^{m_{j}} \in B_{f}$, 
$m_{k} \in \mathbb{Z}$ for each $k$.
An example of such a braid on $f$ strings, corresponding to the element $\sigma_{2} \in B_{f}$, 
$f \geq 3$, is shown in Figure 2.

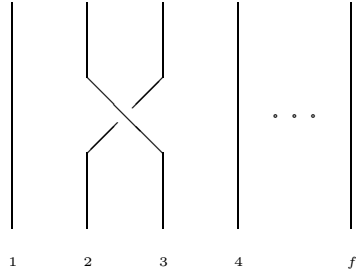
\begin{figure}[hbt]
\label{fig:examplebraid}
\begin{center}
\setlength{\unitlength}{0.5mm}
\begin{picture}(110,70) \large\sf

\put(0,10){\line(0,1){60}}
\put(-1,0){\tiny{\text{$1$}}}

\put(20,10){\line(0,1){20}}
\put(20,50){\line(0,1){20}}
\put(20,50){\line(1,-1){20}}

\put(20,30){\line(1,1){8}}
\put(32,42){\line(1,1){8}}
\put(19,0){\tiny{\text{$2$}}}

\put(40,10){\line(0,1){20}}
\put(40,50){\line(0,1){20}}
\put(39,0){\tiny{\text{$3$}}}

\put(60,10){\line(0,1){60}}
\put(59,0){\tiny{\text{$4$}}}

\put(70,40){\circle{1}} 
\put(75,40){\circle{1}} 
\put(80,40){\circle{1}}

\put(90,10){\line(0,1){60}}
\put(89,0){\tiny{\text{$f$}}}

\end{picture}
\caption{A braid with braid group element $\sigma_{2}$ on $f$ strings}  
\end{center}
\end{figure}

Let $\rho$ be a nontrivial representation of $B_{f}$.
Let $\psi: \rho(B_{f}) \rightarrow \mathbb{C}$ be a functional satisfying the 
three Markov properties:
\begin{eqnarray}
\psi(\rho(\theta_{i}\theta_{j})) & = & \psi(\rho(\theta_{j}\theta_{i})),
 \ \forall \theta_{i}, \theta_{j} \in B_{f},  \label{eq:markov1} \\
\psi(\rho(\theta \sigma_{f-1})) & = & z \psi(\rho(\theta)),  \ \forall \theta \in B_{f-1} \subset B_{f}, \ z \in \mathbb{C}, \label{eq:markov2} \\
\psi(\rho(\theta \sigma_{f-1}^{-1})) & = & 
\widetilde{z}\psi(\rho(\theta)),  \ \forall \theta \in B_{f-1} 
\subset B_{f}, \ \widetilde{z} \in \mathbb{C}, \label{eq:markov3}
\end{eqnarray}
where we take $\psi$ on the right hand sides of Eqs. (\ref{eq:markov2}) and (\ref{eq:markov3}) to be defined
on $\rho(B_{f-1})$ where $B_{f-1}$ is the subgroup of $B_{f}$ generated by $\{\sigma_{i}^{\pm 1}| \  i=1, 2, \ldots f-2\}$. 
Then a link polynomial for $L$ is 
\begin{equation}
\label{eq:link_polynomial}
\widetilde{F}(L) = (z \widetilde{z})^{-(f-1)/2} (\widetilde{z}/z)^{e(b)/2}\psi(\rho(b)),
\end{equation}
where $e(b) = \sum_{k=1}^{j} m_{k}$ \cite{z1}.

\subsection{$U_{q}(osp(1|2n))$ and $U_{-q}(so({2n+1}))$ and their representations}
\label{subsect:reptheory}

\subsubsection{The quantum superalgebra $U_{q}(osp(1|2n))$}
\label{subsubsect:thequantumsuperalgebra}

Let $H^{*}$ be an $n$-dimensional complex vector space with a basis 
$\{ \epsilon_{i} | \ i=1, 2, \ldots, n\}$ and let
$( \cdot, \cdot): H^{*} \rightarrow \mathbb{C}$ be a $\mathbb{C}$-bilinear form defined by
$(\epsilon_{i}, \epsilon_{j}) = \delta_{ij}$ where 
$\delta_{ij} = 1$ if $i=j$ and $0$ otherwise.

Let $\{ \alpha_{i} | \ i=1, 2, \ldots, n\}$ be a basis of simple roots of $H^{*}$: fix 
$\alpha_{i} = \epsilon_{i} - \epsilon_{i+1}$ for $i \leq n-1$ and $\alpha_{n} = \epsilon_{n}$.
The Cartan matrix $A$ for the Lie superalgebra $osp(1|2n)$ 
is identical to that of the Lie algebra $so({2n+1})$:
$A = (A_{ij})_{i,j=1}^{n}$ where $a_{ij} = 2(\alpha_{i},\alpha_{j})/(\alpha_{i},\alpha_{i})$.

The set of positive roots of $osp(1|2n)$ is $\Phi^{+} = \Phi_{0}^{+} \cup \Phi_{1}^{+}$;
$\Phi_{0}^{+} = {\{\epsilon_{i} \pm \epsilon_{j},  2\epsilon_{k} | \ 1 \leq i < j \leq n, 1 \leq k \leq n\}}$
is the set of positive even roots and 
$\Phi_{1}^{+} = {\{\epsilon_{k} | \ 1 \leq k \leq n\}}$
is the set of positive odd roots.

Let $q$ be a non-zero complex parameter satisfying $q^{2} \neq 1$.
The Jimbo quantum superalgebra $U_{q}(osp(1|2n))$ is a $\mathbb{Z}_{2}$-graded Hopf algebra
with generators $\{ e_{i}, f_{i}, k_{i}^{\pm 1} | \ i=1, 2, \ldots, n\}$. 
The grading of each generator is even except for $e_{n}$ and $f_{n}$ which are graded to be odd.
The generators are subject to the following relations:
$$
\displaystyle{e_{i}f_{j} - f_{j}e_{i} = \delta_{ij}\frac{k_{i} - k_{i}^{-1}}{q-q^{-1}}}, 
\hspace{5mm} i < n, 
\hspace{10mm}  
\displaystyle{e_{n}f_{n} + f_{n}e_{n} = \frac{k_{n} - k_{n}^{-1}}{q-q^{-1}}},$$
\begin{equation}
\label{eq:quantumosprelations2}
k_{i}e_{j}k_{i}^{-1} = q^{(\alpha_{i},\alpha_{j})}e_{j}, \hspace{5mm}
k_{i}f_{j}k_{i}^{-1} = q^{-(\alpha_{i},\alpha_{j})}f_{j}, \hspace{10mm} i, j \leq n,
\end{equation}
\begin{equation}
\label{eq:quantumosprelations3}
k_{i}^{\pm 1}k_{j}^{\pm 1} = k_{j}^{\pm 1}k_{i}^{\pm 1}, \hspace{5mm}
k_{i}^{\pm 1}k_{j}^{\mp 1} = k_{j}^{\mp 1}k_{i}^{\pm 1}, \hspace{10mm} i, j \leq n,
\end{equation}
together with the quantum Serre relations which we will not be directly using 
in this paper and which can be found in \cite{zsuper}.

The grading of each graded element $x \in U_{q}(osp(1|2n))$ is indicated  
by writing $[x]=0$ if $x$ is even and $[x]=1$ if $x$ is odd.

We will use the co-algebra structure of $U_{q}(osp(1|2n))$ in dealing with representations of braid groups. 
The co-multiplication is an algebra homomorphism 
$\Delta: U_{q}(osp(1|2n)) \rightarrow U_{q}(osp(1|2n)) \otimes U_{q}(osp(1|2n))$
defined by
$$\Delta(e_{i}) = e_{i} \otimes k_{i} + 1 \otimes e_{i},
\hspace{5mm} \Delta(f_{i}) = f_{i} \otimes 1 + k^{-1}_{i} \otimes f_{i},
\hspace{5mm} \Delta(k_{i}^{\pm 1}) = k_{i}^{\pm 1} \otimes k_{i}^{\pm 1},$$
for $i=1, 2, \ldots, n$, 
and the co-unit $\epsilon: U_{q}(osp(1|2n)) \rightarrow \mathbb{C}$ is a homomorphism defined by
$$\epsilon(e_{i})=\epsilon(f_{i}) = 0, \hspace{5mm} \epsilon(k_{i}^{\pm 1}) = \epsilon(1)=1, 
\hspace{5mm} i=1, 2, \ldots, n.$$

The elements $xy \in U_{q}(osp(1|2n))$ and $x \otimes y \in U_{q}(osp(1|2n)) \otimes U_{q}(osp(1|2n))$ are graded if both elements 
$x, y \in U_{q}(osp(1|2n))$ are graded; in this case
 $[xy] = [x \otimes y] = ([x] + [y]) \pmod{2}$.

$U_{q}(osp(1|2n))$ is a $\mathbb{Z}_{2}$-graded algebra:
$$U_{q}(osp(1|2n)) = \bigoplus_{i=0,1} U_{q}(osp(1|2n))_{i}$$ where 
$U_{q}(osp(1|2n))_{i} = \{ x \in U_{q}(osp(1|2n)) | \ [x]=i \}$. 
An element $x \in U_{q}(osp(1|2n))$ is said to be homogeneous if $x \in \bigcup_{i=0}^{1} U_{q}(osp(1|2n))_{i}$.

There is a graded permutation operator 
$$P: U_{q}(osp(1|2n)) \otimes U_{q}(osp(1|2n)) \rightarrow U_{q}(osp(1|2n)) \otimes U_{q}(osp(1|2n))$$
that acts on homogeneous elements $x, y \in U_{q}(osp(1|2n))$ by:
$$P(x \otimes y) = (-1)^{[y][x]} y \otimes x,$$
the action of which is extended to inhomogeneous elements by linearity.

$U_{q}(osp(1|2n)) \otimes U_{q}(osp(1|2n))$ is a 
$\mathbb{Z}_{2}$-graded algebra with multiplication
$$(a \otimes b) (x \otimes y) = (-1)^{[b][x]} ax \otimes by,$$
for homogeneous elements $a, b, x, y \in U_{q}(osp(1|2n))$ which extends to 
inhomogeneous elements by linearity.

Let $\overline{\pi}_{W}$ be any representation of $U_{q}(osp(1|2n))$ and denote its module by $W$.
The quantum supertrace of $X \in \mathrm{End}_{\mathbb{C}}(W)$ is defined by
$$\mathrm{str}_{q}(X) = \mathrm{str}\big(\overline{\pi}_{W}(k_{2\rho}) \circ X\big),$$
where $\mathrm{str}$ is the usual supertrace and
$k_{2\rho}$ is such a product of the $k_{i}$'s that 
$k_{2\rho}e_{i}k_{2\rho}^{-1} = q^{(2\rho, \alpha_{i})}e_{i}$ for all $i$, where $2\rho \in H^{*}$
is defined by
 $$2\rho = \sum_{\alpha \in \Phi_{0}^{+}} \alpha - \sum_{\beta \in \Phi_{1}^{+}} \beta 
              = \sum_{i=1}^{n} (2n-2i+1) \epsilon_{i}.$$
We define the quantum superdimension of $\overline{\pi}_{W}$ 
to be the quantum supertrace of the identity map on $W$:
$\mathrm{sdim}_{q}(\overline{\pi}_{W}) = \mathrm{str}_{q}(\mathrm{id}_{W})$.

\subsubsection{Representations of $U_{q}(osp(1|2n))$}
\label{subsubsect:repsofthequantumsuperalgebra}

At generic $q$, 
the finite dimensional irreducible representations (irreps) of $U_{q}(osp(1|2n))$ 
are either highest weight deformations of highest weight $osp(1|2n)$ irreps or non-highest weight irreps.

A highest weight $U_{q}(osp(1|2n))$ irrep is completely characterised by its highest weight.
An element $\lambda \in H^{*}$ is said to be integral dominant if 
$l_{i} = 2(\lambda,\alpha_{i})/(\alpha_{i},\alpha_{i}) \in \mathbb{Z}_{+}$ for all $i < n$
and $l_{n} = (\lambda,\alpha_{n})/(\alpha_{n},\alpha_{n}) \in \mathbb{Z}_{+}$. 
The set of all integral dominant weights is the set of highest weights of the highest weight $U_{q}(osp(1|2n))$ irreps and we denote it by $P^{+}$.

We denote the $U_{q}(osp(1|2n))$ irrep with highest weight $\lambda \in P^{+}$ by 
$\overline{\pi}_{\lambda}$ and its corresponding module by $V_{\lambda}$.

\subsubsection{The quantum algebra $U_{q}(so({2n+1}))$}
\label{subsubsect:thequantumalgebra}

The quantum algebra $U_{q}(so({2n+1}))$  is generated by
$\{E_{i}, F_{i}, K_{i}^{\pm 1} | \ i=1, 2, \ldots, n\}$ subject to the relations 
$$E_{i}F_{j} - F_{j}E_{i} = \delta_{ij} \frac{K_{i}-K_{i}^{-1}}{q-q^{-1}}, \hspace{5mm} \forall i, j,$$
the relations (\ref{eq:quantumosprelations2}) and (\ref{eq:quantumosprelations3}) 
replacing $e_{i}$, $f_{i}$ and $k_{i}^{\pm 1}$ with $E_{i}$, $F_{i}$ and $K_{i}^{\pm 1}$, respectively,
and the quantum Serre relations which can be found in \cite{zsuper}.

\subsubsection{Representations of $U_{q}(so({2n+1}))$}
\label{subsubsect:repsofthequantumalgebra}

At generic $q$, each $\lambda \in P^{+}$ is the highest weight of 
a finite dimensional irreducible $U_{-q}(so({2n+1}))$ representation $\pi_{\lambda}$.  
The dimension of $\pi_{\lambda}$ is equal to the dimension of the irreducible $U_{q}(osp(1|2n))$ 
representation $\overline{\pi}_{\lambda}$. 
We also denote the module of $\pi_{\lambda}$ by $V_{\lambda}$. 

Let $V$ denote the module for both the $({2n+1})$-dimensional irreps of 
$U_{-q}(so({2n+1}))$ and $U_{q}(osp(1|2n))$.
Whenever $V$ is considered to be a $U_{q}(osp(1|2n))$-module, 
we takes the grading of its highest weight vector to be {\it odd}.

Let $\pi_{W}$ be a representation of $U_{q}(so({2n+1}))$ with corresponding module $W$ and 
let $X \in \mathrm{End}_{\mathbb{C}}(W)$. 
The quantum trace of $X$ is defined to be $\mathrm{tr}_{q}(X) 
 = \mathrm{tr}\big(\pi_{W}(K_{2\rho}) \circ X\big),$
where $\mathrm{tr}$ is the usual trace and
$K_{2\rho}$ is a product of the $K_{i}$'s such that 
$K_{2\rho}E_{i}K_{2\rho}^{-1} = q^{(2\rho, \alpha_{i})}E_{i}$ for all $i$.
We define the quantum dimension of $\pi_{W}$ to be the quantum trace of the identity map on $W$:
$$\mathrm{dim}_{q}(\pi_{W}) = \mathrm{tr}_{q}(\mathrm{id}_{W}).$$
The $2\rho$ in $U_{q}(so({2n+1}))$ is identical to the $2\rho$ in $U_{q}(osp(1|2n))$.

\subsection{Braid group representations and Markov traces from $U_{q}(osp(1|2n))$ and
$U_{-q}(so({2n+1}))$}
\label{subsect:markovtracesfromalgebras}

\subsubsection{Braid group representations from $U_{q}(osp(1|2n))$ and
$U_{-q}(so({2n+1}))$}

For all integral dominant weights $\lambda$ and $\mu$ there exist invertible maps
$R^{so({2n+1})}_{\lambda \mu} \in \mathrm{End}_{\mathbb{C}}(V_{\lambda} \otimes V_{\mu})$
satisfying 
\begin{eqnarray*}
R^{so({2n+1})}_{\lambda \mu} \cdot (\pi_{\lambda} \otimes \pi_{\mu}) \Delta(x)  & = & 
(\pi_{\lambda} \otimes \pi_{\mu}) \Delta'(x) \cdot R^{so({2n+1})}_{\lambda \mu}, \\
& & \hspace{30mm} \forall x \in U_{-q}(so({2n+1})),
\end{eqnarray*}
where $\Delta' = P \circ \Delta$ is the opposite co-multiplication.
For each such $\lambda$, the map 
\begin{equation}
\label{eq:asdfjkl;}
\check{R}^{so({2n+1})}_{\lambda \lambda} = P \circ R^{so({2n+1})}_{\lambda \lambda}
\end{equation}
commutes with the action of $U_{-q}(so({2n+1}))$:
\begin{eqnarray*}
\check{R}^{so({2n+1})}_{\lambda \lambda} \cdot (\pi_{\lambda} \otimes \pi_{\lambda}) \Delta(x) & = & 
(\pi_{\lambda} \otimes \pi_{\lambda}) \Delta(x) \cdot \check{R}^{so({2n+1})}_{\lambda \lambda}, \\
& & \hspace{30mm} 
\forall x \in U_{-q}(so({2n+1})).
\end{eqnarray*}

Similarly, it was shown in \cite{blumen05} that 
for all integral dominant weights $\lambda$ and $\mu$
there exist maps 
$R^{osp(1|2n)}_{\lambda \mu} \in \mathrm{End}_{\mathbb{C}}(V_{\lambda} \otimes V_{\mu})$ and
\begin{equation}
\label{eq:qwerpoiu}
\check{R}^{osp(1|2n)}_{\lambda \lambda} = P \circ R^{osp(1|2n)}_{\lambda \lambda},
\end{equation}
where $P$ is the graded permutation operator, satisfying
\begin{eqnarray*}
R^{osp(1|2n)}_{\lambda \mu} \cdot (\overline{\pi}_{\lambda} \otimes \overline{\pi}_{\mu}) \Delta(x) & = &  
(\overline{\pi}_{\lambda} \otimes \overline{\pi}_{\mu}) \Delta'(x) \cdot R^{osp(1|2n)}_{\lambda \mu}, \\
\check{R}^{osp(1|2n)}_{\lambda \lambda} 
\cdot (\overline{\pi}_{\lambda} \otimes \overline{\pi}_{\lambda}) \Delta(x) & = &  
(\overline{\pi}_{\lambda} \otimes \overline{\pi}_{\lambda}) \Delta(x) \cdot
 \check{R}^{osp(1|2n)}_{\lambda \lambda}, \\
& & \hspace{30mm} 
\forall x \in U_{q}(osp(1|2n)).
\end{eqnarray*}

We can now define representations of $B_{f}$, the braid group on $f$ strings, in the usual way.
Let $\{\sigma_{i}^{\pm 1} | \ i=1, \ldots, f-1\}$ be the generators of $B_{f}$ as shown in Figure 1. 
Fix $k \geq f$ to be an integer, 
then for each $i=1, \ldots, k-1$ and each $A \in \{so({2n+1}), osp(1|2n)\}$, fix
$$\left(\check{R}^{A}_{\lambda \lambda}\right)_{i}^{\pm 1} = 
\mathrm{id}^{\otimes (i-1)} \otimes 
\left(\check{R}^{A}_{\lambda \lambda}\right)^{\pm 1} \otimes
\mathrm{id}^{\otimes (k-i-1)}.$$
Then the homomorphisms
\begin{equation}
\label{eq:homomorphismdefrep}
\rho^{A}_{\lambda}: \sigma_{i}^{\pm 1} \mapsto 
\left(\check{R}^{A}_{\lambda \lambda}\right)_{i}^{\pm 1},
\end{equation}
define representations of $B_{f}$ in 
$\mathrm{End}_{U_{-q}(so({2n+1}))}(V^{\otimes k})$ and
$\mathrm{End}_{U_{q}(osp(1|2n))}(V^{\otimes k})$.

\subsubsection{Braid group representations from Markov traces on $U_{q}(osp(1|2n))$ and
$U_{-q}(so({2n+1}))$ vector irreps}

We now detail the Markov traces that we define on the representations of $B_{f}$ given in (\ref{eq:homomorphismdefrep}).
For $A = so({2n+1}), osp(1|2n)$, let $\check{R}^{A}_{V V}$ be the map given in (\ref{eq:asdfjkl;})--(\ref{eq:qwerpoiu}) where
\begin{eqnarray}
\label{eq:Rmatrixso}
 & & 
\check{R}^{so({2n+1})}_{V V} \in \mathrm{End}_{U_{-q}(so({2n+1}))}(V \otimes V),  \\
\label{Rmatrixosp}
 & & \check{R}^{osp(1|2n)}_{V V} \in \mathrm{End}_{U_{q}(osp(1|2n))}(V \otimes V),
\end{eqnarray}
and define ${\cal{C}}_{f}^{A}$ to be the complex algebra generated by 
$\left\{\left(\check{R}^{A}_{V V}\right)_{i}^{\pm 1} \Big| \ i=1, 2, \ldots, f-1 \right\}$.
In addition, define the maps $\psi^{A}: {\cal{C}}_{f}^{A} \rightarrow \mathbb{C}$ by
\begin{eqnarray}
\psi^{so({2n+1})}(X_{1}) & = & \frac{\mathrm{tr}_{-q}(X_{1})} {\big(\mathrm{dim}_{-q}(V) \big)^{f}}, \hspace{5mm}
\forall X_{1} \in {\cal{C}}_{f}^{so({2n+1})}, \label{eq:somarkovtrace} \\
\psi^{osp(1|2n)}(X_{2}) & = & \frac{\mathrm{str}_{q}(X_{2})} {\big(\mathrm{sdim}_{q}(V) \big)^{f}},  \hspace{5mm}
\forall X_{2} \in {\cal{C}}_{f}^{osp(1|2n)}, \label{eq:ospmarkovtrace}
\end{eqnarray}
where we recall that $\mathrm{dim}_{-q}(V)= \frac{-q^{2n}+q^{-2n}}{q-q^{-1}}+1$ is the quantum dimension of the
$U_{-q}(so({2n+1}))$-module $V$ and note that $\mathrm{sdim}_{q}(V) = \mathrm{dim}_{-q}(V)$.

It is well known that $\psi^{so({2n+1})}$ is a Markov trace and it was shown in \cite{blumen05}
that $\psi^{osp(1|2n)}$ is also a Markov trace, i.e. both $\psi^{so({2n+1})}$ and $\psi^{osp(1|2n)}$ satisfy Eqs. (\ref{eq:markov1})--(\ref{eq:markov3}) upon substituting them for $\psi$ \cite{blumen05, tw}.

\subsection{Quantum link invariants $\Phi^{so({2n+1})}_{L}(-q)$ and $\Phi^{osp(1|2n)}_{L}(q)$}

As $\psi^{so({2n+1})}$ and $\psi^{osp(1|2n)}$ are Markov traces,
Eq. (\ref{eq:link_polynomial}) defines the uncoloured quantum link invariants $\Phi^{so({2n+1})}_{L}(-q)$ and $\Phi^{osp(1|2n)}_{L}(q)$ obtained by substituting $\psi^{so({2n+1})}$ and 
$\psi^{osp(1|2n)}$ for $\psi$ in Eq. (\ref{eq:link_polynomial}), respectively, and using the representations of the braid group given in Eq. (\ref{eq:homomorphismdefrep}).

In the next section we define the Kauffman polynomial  
from the unspecialised Birman--Wenzl--Murakami algebra $BWM_{f}$ and detail the connections 
between the Kauffman link invariant and the quantum link invariants 
$\Phi^{so({2n+1})}_{L}(-q)$ and $\Phi^{osp(1|2n)}_{L}(q)$.

\section{Birman--Wenzl--Murakami algebras and the Kauffman link polynomial}
We now discuss the unspecialised and specialised 
Birman--Wenzl--Murakami algebras \cite{bw,murakami}. 
These algebras have natually defined trace functionals that we use to define
the Kauffman link invariant \cite{w2}.

\subsection{Unspecialised Birman--Wenzl--Murakami algebra $BWM_{f}$}
\subsubsection{Definition}
Let $f \geq 2$ be an integer and $r, s$ indeterminates.
Let $\mathbb{C}(r,s)$ be the field of rational polynomials in $r$ and $s$ with complex coefficients.
The unspecialised Birman--Wenzl--Murakami algebra
$BWM_{f}$ \cite{w2} is an associative algebra taken over $\mathbb{C}(r,s)$ that is generated by 
the invertible elements $\{ g_{i} | \ 1 \leq i \leq f-1 \}$ subject to the relations
\newline
$$
\begin{array}{ll}
g_{i} g_{j} = g_{j} g_{i}, 				 & |i-j| >1, 	\\
g_{i} g_{i+1} g_{i} = g_{i+1} g_{i} g_{i+1}, 	 & 1 \leq i \leq f-2, 	\\
e_{i} g_{i} = r^{-1} e_{i},                      & 1 \leq i \leq f-1, 	\\
e_{i}g_{i-1}^{\pm 1}e_{i} = r^{\pm 1} e_{i}, 	 & 1 \leq i \leq f-1,
\end{array} $$
\newline
where each $e_{i}$ is defined by
$$(s-s^{-1})(1-e_{i}) = g_{i}-g_{i}^{-1}, \hspace{10mm}  1 \leq i \leq f-1.$$
Each $g_{i}$ also satisfies
$(g_{i} - r^{-1})(g_{i} + s^{-1})(g_{i} - s) = 0$.

\subsubsection{Trace functional on $BWM_{f}$}
$BWM_{f}$ is equipped with a functional 
$\mathrm{tr}: BWM_{f} \rightarrow \mathbb{C}(r,s)$ satisfying \cite{w2}:
\begin{equation}
\label{eq:tracefunctionalnew}
\mathrm{tr}(a \chi b) = \mathrm{tr}(\chi) \mathrm{tr}(ab), 
\hspace{10mm} \forall a, b \in BWM_{f-1}, \hspace{5mm} \chi \in \{g_{f-1}, e_{f-1}\},
\end{equation}
where we regard each element of $BWM_{f-1}$ as an element of $BWM_{f}$
under the canonical inclusion (i.e. we take $g_{i} \in BWM_{f-1}$ as an element of
$BWM_{f}$ via ${g_{i} \hookrightarrow g_{i}}$). 
We recall the definition of the trace functional $\mathrm{tr}$ 
in section \ref{sec:technicalresults} 
and the well-known result that we can use $\mathrm{tr}$ to construct link invariants in 
subsection \ref{sec:linkinvariants}.

\subsection{Link invariants from $BWM_{f}$}
\label{sec:linkinvariants}

The Kauffman link invariant can be defined using $BWM_{f}$ as follows \cite{w2}.
Let a link $L$ be presented as a braid on $f$ strings
with corresponding braid group element 
$b = \sigma_{i_{1}}^{m_{1}} \sigma_{i_{2}}^{m_{2}} \cdots \sigma_{i_{j}}^{m_{j}} \in B_{f}$, 
where $m_{k} \in \mathbb{Z}$ for each $k$,
 and let $\beta \in BWM_{f}$ be the image of $b$ under the homomorphism
$\sigma_{i}^{\pm 1} \mapsto g_{i}^{\pm 1}$.
Define $e(b) = \sum_{k} m_{k}$ and
$\widehat{\mathrm{tr}}(b) = r^{-e(b)} \mathrm{tr}(\beta)$, then the Kauffman link invariant 
of $L$ is
\begin{equation}
\label{definitionoflininvariant} 
F_{L}(r,s) = \widehat{\mathrm{tr}}(\sigma_{1})^{1-f}\widehat{\mathrm{tr}}(b).
\end{equation}

We now show that $F_{L}(r,s)$ is a particular example of the link invariant $\widetilde{F}(L)$ defined using (\ref{eq:link_polynomial}).
The homomorphism $\rho: \sigma_{i}^{\pm 1} \mapsto g_{i}^{\pm 1} \in BWM_{f}$ yields a representation
of $B_{f}$, and $\mathrm{tr} \circ \rho$ is a functional satisfying the three Markov properties as follows (corresponding to Eqs. (\ref{eq:markov1})--(\ref{eq:markov3})):
\begin{itemize}
\item[(i)] $\mathrm{tr}(ab) =\mathrm{tr}(ba)$, 
$\forall a, b \in BWM_{f}$,
\item[(ii)] $\mathrm{tr}(ag_{f-1}) = \mathrm{tr}(g_{f-1})\mathrm{tr}(a) = \frac{r}{x} \mathrm{tr}(a)$, 
$\forall a \in BWM_{f-1}$,
\item[(iii)] $\mathrm{tr}(ag_{f-1}^{-1}) = \mathrm{tr}(g_{f-1}^{-1})\mathrm{tr}(a) = \frac{r^{-1}}{x} \mathrm{tr}(a)$, 
$\forall a \in BWM_{f-1}$.
\end{itemize}
It then follows that $\mathrm{tr} \circ \rho$ can be used to construct a link invariant $\widetilde{F}(L)$ following (\ref{eq:link_polynomial}): 
let $L$ be a link that is presented as the closure of a braid with corresponding braid group element $b \in B_{f}$. Then the link invariant $\widetilde{F}(L)$ for $L$ is
$$\widetilde{F}(L) = x^{f-1}r^{-e(b)} \mathrm{tr}(\rho(b)),$$
which equals the right hand side of (\ref{definitionoflininvariant}).

\subsection{Specialised Birman--Wenzl--Murakami algebra $BWM_{f}(t,q)$}
\subsubsection{Definition and trace functional}
We denote by $BWM_{f}(t,q)$ the algebra obtained by formally replacing the indeterminates $r$ and $s$ in $BWM_{f}$ with the complex numbers $t$ and $q$, respectively.

$BWM_{f}(t,q)$ is equipped with a functional 
$\mathrm{tr}: BWM_{f}(t,q) \rightarrow \mathbb{C}$ satisfying
\begin{equation}
\label{eq:tracefunctional}
\mathrm{tr}(a \chi b) = \mathrm{tr}(\chi) \mathrm{tr}(ab), 
\hspace{10mm} \forall a, b \in BWM_{f-1}(t,q), \hspace{5mm} \chi \in \{g_{f-1}, e_{f-1}\},
\end{equation}
where we regard each element of $BWM_{f-1}(t,q)$ as an element of $BWM_{f}(t,q)$
under the canonical inclusion.

\subsubsection{Representations of $BWM_{f}(-q^{2n},q)$ and $BWM_{f}(q^{2n},-q)$ from 
$U_{q}(osp(1|2n))$ and $U_{q}(so({2n+1}))$}

Certain representations of $U_{q}(sp(2n))$, $U_{q}(so({2n+1}))$ and $U_{q}(osp(1|2n))$ 
yield representations of different specialisations of $BWM_{f}(t,q)$. 
In this subsection, we recall how
representations of $U_{q}(osp(1|2n))$ yield representations of
$BWM_{f}(-q^{2n},q)$ \cite{blumen05} and how representations of $U_{-q}(so({2n+1}))$ yield 
representations of $BWM_{f}(q^{2n},-q)$ \cite{tw}. 

Fix $q$ to be generic and non-zero in this rest of this section and in 
section \ref{subsection:connectionbetween}. Recall the maps 
$\check{R}^{so({2n+1})}_{V V}$ and $\check{R}^{osp(1|2n)}_{V V}$ 
from (\ref{eq:Rmatrixso}) and (\ref{Rmatrixosp}), respectively.
The homomorphism
$$\displaystyle{ \rho^{so({2n+1})}_{V}: g_{i}^{\pm 1} \mapsto 
\left(\check{R}^{so({2n+1})}_{V V}\right)_{i}^{\pm 1} }$$
yields a representation of $BWM_{f}(q^{2n},-q)$ \cite{tw},
and the homomorphism
$$\displaystyle{ \Upsilon: g_{i}^{\pm 1} \mapsto
-\left(\check{R}^{osp(1|2n)}_{V V}\right)_{i}^{\pm 1} }$$
yields a representation of $BWM_{f}(-q^{2n},q)$ \cite{blumen05}.

\subsubsection{Kauffman link invariants from $BWM_{f}(-q^{2n},q)$ and $BWM_{f}(q^{2n},-q)$}

Kauffman link invariants $F_{L}(-q^{2n},q)$ and $F_{L}(q^{2n},-q)$
can be respectively defined from $BWM_{f}(-q^{2n},q)$ and $BWM_{f}(q^{2n},-q)$ following (\ref{definitionoflininvariant}). 
The only matters that need to be considered are that 
the image $\beta(-q^{2n},q)$ of $b \in B_{f}$ under the homomorphism
$\sigma_{i}^{\pm 1} \mapsto g_{i}^{\pm 1} \in BWM_{f}(-q^{2n},q)$ is well-defined as is
the image $\beta(q^{2n},-q)$ of $b$ under the homomorphism
$\sigma_{i}^{\pm 1} \mapsto g_{i}^{\pm 1} \in BWM_{f}(q^{2n},-q)$, both of which are true.

\subsection{Connections between $\Phi^{osp(1|2n)}_{L}(q)$,
$\Phi^{so({2n+1})}_{L}(-q)$  and $F_{L}(-q^{2n},q)$, $F_{L}(q^{2n},-q)$}
\label{subsection:connectionbetween}
Recall the definitions of the Markov traces 
$\psi^{so({2n+1})}$ and $\psi^{osp(1|2n)}$ from (\ref{eq:somarkovtrace})--(\ref{eq:ospmarkovtrace}) and the trace functional $\mathrm{tr}$ on $BWM_{f}(-q^{2n},q)$ and $BWM_{f}(q^{2n},-q)$. 

We can now prove Theorem \ref{th:quantuminvariantequaltokauffman}, which states that
for an arbitrary link $L$ and each positive integer $n$,
\begin{itemize}
\item[(i)] $\Phi^{osp(1|2n)}_{L}(q) = F_{L}(-q^{2n},q)$, and 
\item[(ii)] $\Phi^{so({2n+1})}_{L}(-q) = F_{L}(q^{2n},-q)$.
\end{itemize}

\begin{proof}
Eq. (\ref{eq:A}) was proved in \cite{blumen05} and Eq.(\ref{eq:B}) is well-known:
\begin{eqnarray}
\psi^{osp(1|2n)} \big(\Upsilon(a) \big) & = & \mathrm{tr}(a),
 \hspace{10mm} \forall a \in BWM_{f}(-q^{2n},q), \label{eq:A} \\
\psi^{so({2n+1})} \left(\rho^{so({2n+1})}_{V}(a) \right) & = & \mathrm{tr}(a), 
\hspace{10mm} \forall a \in BWM_{f}(q^{2n},-q). \label{eq:B}
\end{eqnarray}
We firstly prove (i) of the theorem, the proof of (ii) is similar and will be omitted.
The invariant $F_{L}(-q^{2n},q)$ arises from applying  
$\mathrm{tr} \circ \rho(-q^{2n},q)$ in (\ref{eq:link_polynomial}) where
$\rho(-q^{2n},q): B_{f} \rightarrow BWM_{f}(-q^{2n},q)$ is a representation defined by the homomorphism $\sigma_{i}^{\pm 1} \mapsto g_{i}^{\pm 1}$ and $\mathrm{tr}$ is the trace 
functional on $BWM_{f}(-q^{2n},q)$.

The quantum link invariant $\Phi^{osp(1|2n)}_{L}(q)$ arises from applying 
$\psi^{osp(1|2n)} \big(\Upsilon(a) \big) \circ \rho(-q^{2n},q)$ in (\ref{eq:link_polynomial}).
The proof of Theorem \ref{th:quantuminvariantequaltokauffman}(i) then follows from (\ref{eq:A}).

\end{proof}
 
Restating Theorem \ref{th:quantuminvariantequaltokauffman}, the quantum link invariant $\widetilde{F}(L)$ obtained by colouring each component
of the link $L$ with the $({2n+1})$-dimensional irreducible representation of $U_{q}(osp(1|2n))$ 
(resp. $U_{-q}(so({2n+1}))$) 
is {\it identical} to the Kauffman link invariant $F_{L}(-q^{2n},q)$ (resp. $F_{L}(q^{2n},-q)$).

We now prove Theorem \ref{th:equalityofkauffman}, which states: 
Let $L(m)$ be a link presented as the canonical closure of a braid
with corresponding braid group element $(\sigma_{1})^{m}$ for $m \in \mathbb{Z}$.
Then $F_{L(m)}(-r,-s) = F_{L(m)}(r,s)$ and $\Phi^{osp(1|2n)}_{L(m)}(q)=\Phi^{so({2n+1})}_{L(m)}(-q)$.
\begin{proof}
Lemma \ref{lem:positivenegativelinkpolynomials} gives 
$F_{L(m)}(-r,-s) = F_{L(m)}(r,s)$ and Theorem \ref{th:quantuminvariantequaltokauffman} completes the proof.
\end{proof}

\begin{lemma}
\label{lem:positivenegativelinkpolynomials}
$F_{L(m)}(-r,-s) = F_{L(m)}(r,s)$
for each link $L(m), m \in \mathbb{Z}$, where $L(m)$ is the closure of a braid with corresponding
braid group element $(\sigma_{1})^{m}$.
\end{lemma}
\begin{proof}
From Eq. (\ref{definitionoflininvariant}), 
$F_{L(m)}(r,s) = \widehat{\mathrm{tr}}(\sigma_{1})^{1-f}\widehat{\mathrm{tr}}\left((\sigma_{1})^{m}\right)$.
Firstly, note that $\widehat{\mathrm{tr}}(\sigma_{1}) = x^{-1}$ and that 
$x \in BWM_{f}(-r,-s)$ is identical to $x \in BWM_{f}(r,s)$ when both are considered as elements of $\mathbb{C}(r,s)$.
Now $\widehat{\mathrm{tr}}\left((\sigma_{1})^{m}\right) = r^{-m} \mathrm{tr}( (g_{1})^{m} )$ and from  
Lemma \ref{lem:powersofg} we have:
$$\mathrm{tr}( (g_{1})^{m} ) = a_{m}(r,s) + b_{m}(r,s) r x^{-1} + c_{m}(r,s) x^{-1}.$$
By inspection, 
\begin{eqnarray*}
a_{m}(-r,-s) & = & (-1)^{m} a_{m}(r,s), \\
b_{m}(-r,-s) & = & (-1)^{m+1} b_{m}(r,s), \\
c_{m}(-r,-s) & = & (-1)^{m} c_{m}(r,s),
\end{eqnarray*}
and it follows that
\begin{eqnarray*}
F_{L(m)}(-r,-s) & = & x^{f-1} (-r)^{-m} \left( a_{m}(-r,-s) + b_{m}(-r,-s)(-r) x^{-1} + c_{m}(-r,-s) x^{-1} \right) \\ 
                & = & F_{L(m)}(r,s).
\end{eqnarray*}

\end{proof}

\section{The relationship between $F_{L}(-q^{2n},q)$ and $F_{L}(q^{2n},-q)$}
\label{sec:relationshipbetween}

Theorem \ref{th:equalityofkauffman} shows that $F_{L(m)}(-q^{2n},q)=F_{L(m)}(q^{2n},-q)$ for all links $L(m)$.
We now prove Theorem \ref{thm:equalsubstitutionofvariables}, which states:
for each arbitrary link $L$ and each positive integer $n$,
 $\Phi^{osp(1|2n)}_{L}(q)$ and $\Phi^{so({2n+1})}_{L}(-q)$ are equal up a substitution of variables.

\begin{proof}
We prove the result by showing that
$F_{L}(-q^{2n},q)$ can in principle be obtained from $F_{L}(q^{2n},-q)$ (and vice-versa) for all links $L$ by a substitution of variables.
The result for $\Phi^{osp(1|2n)}_{L}(q)$ and $\Phi^{so({2n+1})}_{L}(-q)$ then follows from Theorem \ref{th:quantuminvariantequaltokauffman}.

In this proof we refer to Bratteli diagrams for $BMW_{f}$ and related concepts but 
leave the detail of these to section \ref{sec:brattelidiagrams} as their explanation is lengthy.

We fix $\Omega_{f}$ to be the set of pairs $(R,S)$ of paths of length $f$ 
in the Bratteli diagram for $BWM_{f}$ where $shp(R)=shp(S)$, 
where $shp(R)$ is the Young diagram on the $f^{th}$ level of
the Bratteli diagram for $BWM_{f}$ at which the path $R$ ends.
Ram and Wenzl wrote down  
an explicit basis $\{E_{ST} \in BWM_{f} | \ (S,T) \in \Omega_{f} \}$ of $BWM_{f}$ \cite{rw}.
This basis is a set of matrix units, i.e. the basis elements satisfy
$E_{QR}E_{ST} = \delta_{RS} E_{QT}$.

Recall from (\ref{definitionoflininvariant}) that $F_{L}(r,s)$ is obtained by multiplying together weighted traces of
particular elements of $BWM_{f}$. Writing each element $X \in BWM_{f}$ as a linear combination of matrix units:
$$\displaystyle{X = \sum_{(S,T) \in \Omega_{f}} c_{ST} E_{ST}, \ 
c_{ST} \in \mathbb{C}(r,s)},$$
the trace of $X$ is 
$$\displaystyle{ \mathrm{tr}(X) = \sum_{(S,T) \in \Omega_{f}} c_{ST} \mathrm{tr}(E_{ST}) },$$
where $\mathrm{tr}(E_{SS}) \neq 0$ for all $(S,S) \in \Omega_{f}$ from Lemma \ref{lemtraceofmatrixunit}.
Given such an element $X$, we fix
$$\displaystyle{X(t,q) = \sum_{(S,T) \in \Omega_{f}} c_{ST}(t, q) E_{ST}(t, q) }$$
to be the corresponding element of
$BWM_{f}(t,q)$ obtained by replacing the indeterminates $r$ and $s$ in $X$ with the complex numbers
$t$ and $q$, respectively.

We fix  
$\Omega_{f}(-q^{2n},q)$ to be the set of pairs $(R,S)$ 
of paths of length $f$ in the truncated Bratteli diagram for the semisimple algebra
$BWM_{f}(-q^{2n},q)/J_{f}(-q^{2n},q)$ where $shp(R)=shp(S)$.  We similarly define $\Omega_{f}(q^{2n},-q)$.

Note that $\Omega_{f}(-q^{2n},q)=\Omega_{f}(q^{2n},-q)$, which arises from
the result in Lemma \ref{lem:JCBsong} that $Q_{\lambda}(-q^{2n},q) = Q_{\lambda}(q^{2n},-q)$.

Let $X \in BWM_{f}$ be any element where each of 
$X(-q^{2n},q)$ and $X(q^{2n},-q)$ is well-defined.
Then
\begin{eqnarray}
\mathrm{tr}(X)\big|_{(r,s)=(-q^{2n},q)} = \mathrm{tr}\big(X(-q^{2n},q)\big) 
 & = & \sum_{(S,T) \in \Omega_{f}(-q^{2n},q)} 
c_{ST}(-q^{2n},q) \mathrm{tr}(E_{ST}(-q^{2n},q)) \nonumber \\
 & = & \left. \left( \sum_{(S,T) \in \Omega_{f}(-q^{2n},q)} c_{ST} \mathrm{tr}(E_{ST})  \right) 
\right|_{(r,s)=(-q^{2n},q)} \label{eq:1stequation}
\end{eqnarray}
as $\mathrm{tr}(E_{SS}) \big|_{(r,s)=(-q^{2n},q)} = 0$ if 
$(S,S) \notin \Omega_{f}(-q^{2n},q)$, and similarly
\begin{equation}
\label{eq:2ndequation}
\mathrm{tr}(X)\big|_{(r,s)=(q^{2n},-q)} = 
\mathrm{tr}\big(X(q^{2n},-q)\big) 
 = \left. \left( \sum_{(S,T) \in \Omega_{f}(q^{2n},-q)} 
c_{ST} \mathrm{tr}(E_{ST})  \right) 
\right|_{(r,s)=(q^{2n},-q)}.
\end{equation}
Note that the sums on the right hand sides of (\ref{eq:1stequation}) and 
(\ref{eq:2ndequation}) are over the same sets. 

We rewrite parts of Eqs. (\ref{eq:1stequation}) and (\ref{eq:2ndequation}):
\begin{eqnarray}
\nonumber
\left. \left( \sum_{(S,T) \in \Omega_{f}(-q^{2n},q)} c_{ST} \mathrm{tr}(E_{ST})  \right) 
\right|_{(r,s)} & = & 
\left\{ 
\begin{array}{ll}
\displaystyle{\mathrm{tr}\big(X(-q^{2n},q)\big)}, & \displaystyle{\mbox{if } (r,s)=(-q^{2n},q)}, \\
\displaystyle{\mathrm{tr}\big(X(q^{2n},-q)\big)}, & \displaystyle{\mbox{if } (r,s)=(q^{2n},-q)}.
\end{array}
\right. \\
 & & \label{eq:leftandright}
\end{eqnarray}
It follows from (\ref{eq:leftandright}) that it is possible in principle to obtain 
$\mathrm{tr}(X(q^{2n},-q))$ from $\mathrm{tr}(X(-q^{2n},q))$ by applying the mapping
 $(-q^{2n},q) \mapsto (q^{2n},-q)$
(and similarly possible to obtain $\mathrm{tr}(X(-q^{2n},q))$ from $\mathrm{tr}(X(q^{2n},-q))$ by applying the reverse mapping).
However, it may be difficult to do this in practise as $q^{2n}$ and $q$ are not independent:  
the substitution can be expressed as the mapping $q^{m} \mapsto (-q)^{m}$ and $-q^{2n} \mapsto q^{2n}$, however  
the first mapping also gives $(q)^{2n} \mapsto (-q)^{2n}$. 
It follows that the substitution can be directly done if $q^{2n}$ does not appear in $\mathrm{tr}(X(-q^{2n},q))$ 
or if the left hand side of Eq. (\ref{eq:leftandright}) is explicitly known.  Similar considerations hold for 
applying the mapping $(q^{2n},-q) \mapsto (-q^{2n},q)$ to $\mathrm{tr}(X(q^{2n},-q))$ to obtain
$\mathrm{tr}(X(-q^{2n},q))$.

It follows that there is an abstract symmetry between 
$F_{L}(-q^{2n},q)$ and $F_{L}(q^{2n},-q)$ given by mapping between the relevant pairs of signed powers of $q$.
However, it may not be possible to directly obtain one of the invariants from the other 
by applying the relevant mappings without additional knowledge of the traces of certain elements in $BMW_{f}$.
\end{proof}

Similar results will hold for any Kauffman link invariants $F_{L}(r,s)$ and $F_{L}(r',s')$
where the truncated Bratteli diagrams for the relevant semisimple quotients of $BWM_{f}(r,s)$ and
$BWM_{f}(r',s')$ are identical.

\subsection{The case for $q$ a root of unity}
We have not considered the relationship between the relevant quantum link invariants
when $q$ is a root of unity. However, we believe that 
similar results hold for $q$ a root of unity as at generic $q$.
For $q$ a root of unity, the truncated Bratteli diagram for
$BWM_{f}(\mp q^{2n},\pm q)/J_{f}(\mp q^{2n},\pm q)$ is, for a sufficiently large $f$ 
(depending on the root of unity), a proper subgraph
of the truncated Bratteli diagram for $BWM_{f}(\mp q^{2n},\pm q)/J_{f}(\mp q^{2n},\pm q)$ at generic $q$ \cite{blumen05,w2}.
The fact that this subgraph is proper is intimately related to the existences of the 
truncated dominant Weyl alcoves in the relevant weight spaces of $U_{q}(osp(1|2n))$ and
$U_{-q}(so({2n+1}))$ for $q$ a root of unity.

\section{Bratteli diagrams for Birman--Wenzl--Murakami algebras}
\label{sec:brattelidiagrams}

\subsection{Bratteli diagram for $BWM_{f}$}
Following \cite{w2} we say that an algebra $A$ is semisimple if it is isomorphic
to a direct sum of matrix algebras: $A \cong \bigoplus_{i} M_{k_{i}}(\mathbb{C})$, where
$k_{i} \in \{1, 2, \ldots \}$ and
$M_{j}(\mathbb{C})$ is the algebra of $j \times j$ matrices with complex entries. 
$BWM_{f}$ is semisimple \cite{w2} and its structure can be conveniently represented 
by a Bratteli diagram, which is
an undirected graph encoding information about
a sequence $\mathbb{C} \cong A_{0} \subset A_{1} \subset A_{2} \subset \cdots$ of 
inclusions of finite dimensional semisimple algebras \cite{rw}. 
 
To draw a Bratteli diagram for $BWM_{f}$, 
we firstly need the Young lattice \cite{w2},
which is (almost) identical to the Bratteli diagram for the 
sequence of inclusions of group algebras of the symmetric group:
$\mathbb{C}S_{1} \subset \mathbb{C}S_{2} \subset  \mathbb{C}S_{3} \subset \cdots $.

The vertices of the Young lattice are grouped into levels:
\begin{itemize}
\item[(i)]
each Young diagram with $f \geq 0$ boxes labels a vertex on the $f^{th}$ level of the Young lattice,
\item[(ii)] a vertex $\lambda$ on the $f^{th}$ level is connected to a vertex $\mu$ on the 
$(f+1)^{st}$ level by an edge if and only if $\lambda$ and $\mu$ differ by exactly one box, and
\item[(iii)] the empty Young diagram (containing no boxes) is on the $0^{th}$ level.
\end{itemize}

For each $f$, let $Y_{f}$ be the set of vertices on the $f^{th}$ level of the Young lattice and define
$\displaystyle{ \Gamma_{f} = \bigcup_{f - 2k \geq 0} Y_{f-2k} }$ where $k$ ranges over 
all of $\mathbb{Z}_{+}$.  
$\Gamma_{k}$ is the set of vertices on the $k^{th}$ level of the Bratteli diagram for $BWM_{f}$.
A vertex $\lambda$ on the $k^{th}$ level is connected to a vertex $\mu$ on the $(k+1)^{st}$ level if 
and only if $\lambda$ and $\mu$ differ by exactly one box.
We show the Bratteli diagram for $BWM_{f}$ up to the $4^{th}$ level in
Figure 3.

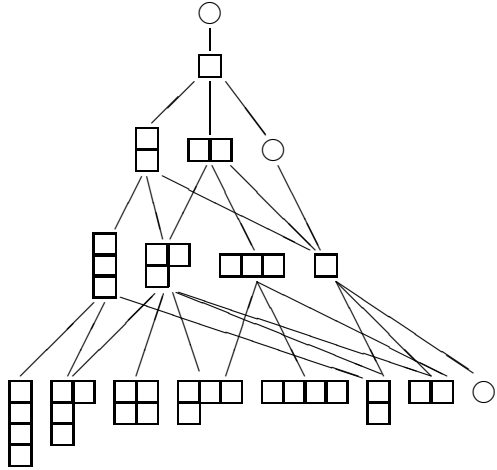
\begin{figure}[hbt]
\label{fig:BWbrattelidiagram}
\begin{center}
\setlength{\unitlength}{0.14mm}
\begin{picture}(460,440) \large\sf
\thinlines

\put(190,430){\circle{20}} 

\put(190,395){\line(0,1){20}} 

\put(180,370){\framebox(20,20){}}

\put(120,300){\framebox(20,20){}}
\put(120,280){\framebox(20,20){}}

\put(170,290){\framebox(20,20){}}
\put(190,290){\framebox(20,20){}}

\put(250,300){\circle{20}} 

\put(80,160){\framebox(20,20){}}
\put(80,180){\framebox(20,20){}}
\put(80,200){\framebox(20,20){}}
\put(100,225){\line(1,2){25}} 
\put(130,275){\line(1,-4){15}} 

\put(130,170){\framebox(20,20){}}
\put(130,190){\framebox(20,20){}}
\put(150,190){\framebox(20,20){}}
\put(152,215){\line(1,2){35}} 
\put(192,285){\line(1,-2){40}}

\put(200,180){\framebox(20,20){}}
\put(220,180){\framebox(20,20){}}
\put(240,180){\framebox(20,20){}}

\put(290,180){\framebox(20,20){}}
\put(210,285){\line(1,-1){80}} 
\put(255,285){\line(1,-2){40}} 
\put(145,275){\line(2,-1){140}} 

\put(0,0){\framebox(20,20){}}
\put(0,20){\framebox(20,20){}}
\put(0,40){\framebox(20,20){}}
\put(0,60){\framebox(20,20){}}

\put(10,85){\line(1,1){70}} 
\put(135,325){\line(1,1){40}} 
\put(190,315){\line(0,1){50}} 
\put(205,365){\line(3,-4){38}} 

\put(60,60){\framebox(20,20){}}
\put(40,20){\framebox(20,20){}}
\put(40,40){\framebox(20,20){}}
\put(40,60){\framebox(20,20){}}
\put(55,85){\line(1,2){35}} 
\put(60,85){\line(1,1){78}} 

\put(100,60){\framebox(20,20){}}
\put(100,40){\framebox(20,20){}}
\put(120,60){\framebox(20,20){}}
\put(120,40){\framebox(20,20){}}
\put(120,85){\line(1,3){26}} 

\put(160,60){\framebox(20,20){}}
\put(160,40){\framebox(20,20){}}
\put(180,60){\framebox(20,20){}}
\put(200,60){\framebox(20,20){}}
\put(155,165){\line(1,-3){25}} 
\put(205,85){\line(1,3){30}} 
\put(105,160){\line(3,-1){230}} 

\put(240,60){\framebox(20,20){}}
\put(260,60){\framebox(20,20){}}
\put(280,60){\framebox(20,20){}}
\put(300,60){\framebox(20,20){}}
\put(235,175){\line(1,-2){45}} 

\put(340,60){\framebox(20,20){}}
\put(340,40){\framebox(20,20){}}
\put(158,165){\line(5,-2){195}} 
\put(310,175){\line(1,-1){90}} 
\put(310,175){\line(1,-2){45}} 
\put(310,175){\line(3,-2){130}} 

\put(380,60){\framebox(20,20){}}
\put(400,60){\framebox(20,20){}}
\put(160,165){\line(3,-1){240}} 
\put(235,175){\line(2,-1){180}} 

\put(450,70){\circle{20}} 

\end{picture}
\caption{The Bratteli diagram for $BWM_{f}$ up to the $4^{th}$ level inclusive}  
\end{center}
\end{figure}

\subsection{A basis of $BWM_{f}$ and matrix units for $BWM_{f}$}
Ram and Wenzl wrote down \cite{rw} 
an explicit basis $\{E_{ST} \in BWM_{f} | \ (S,T) \in \Omega_{f} \}$ of $BWM_{f}$, this notation
we explain below.
This basis is a set of matrix units, i.e. the basis elements satisfy
$E_{QR}E_{ST} = \delta_{RS} E_{QT}$.

We say that $R$ is a path of length $f$ in
the Bratteli diagram for $BWM_{f}$ 
if 
\begin{itemize}
\item[(i)]
$R = (r_{0}, r_{1}, \ldots, r_{f})$ is a sequence of $f+1$ Young diagrams
where $r_{k} \in \Gamma_{k}$ for each $k = 0, 1, \ldots, f$, and 
\item[(ii)] $r_{i}$ is connected by an edge to the vertex $r_{i+1}$ in the Bratteli diagram for 
$BWM_{f}$ for each $i=0, 1, \ldots, f-1$. 
\end{itemize}
We write $shp(R) = r_{f}$ and fix
$\Omega_{f}$ be the set of pairs $(R,S)$ of paths of length $f$ 
in the Bratteli diagram for $BWM_{f}$ where $shp(R)=shp(S)$.

The following lemma \cite[Lemma 4.2]{w2} is important.

\begin{lemma} 
\label{lemtraceofmatrixunit}
let $R$ be a path of length $f$ in the Bratteli diagram for $BWM_{f}$
where $shp(R) = \lambda$.
Then $\mathrm{tr}(E_{RR}) = Q_{\lambda}(r,s)/x^{f}$, 
where $Q_{\lambda}(r,s)$ is the function
given in (\ref{eq:themagixQpolynomial}) and $\displaystyle{x = \frac{r-r^{-1}}{s-s^{-1}} + 1}$.
\end{lemma}

\subsection{Semisimple quotients of $BWM_{f}(t,q)$}
Define an ideal $J_{f}(t,q) \subset BWM_{f}(t,q)$ with respect to 
$\mathrm{tr}$ by
$$J_{f}(t,q) = \left\{b \in BWM_{f}(t,q) | \ 
\mathrm{tr}(ab) = 0, \ \forall a \in BWM_{f}(t,q)\right\}.$$
If $q$ is nonzero and not a root of unity and $t \neq \pm q^{k}$ for any $k \in \mathbb{Z}$, then
$J_{f}(t,q)=0$ and $BWM_{f}(t,q)$ is semisimple.
If $q$ is not a root of unity and $t = \pm q^{k}$ for some $k \in \mathbb{Z}$, then
$J_{f}(\pm q^{k},q)$ may be non-zero and
the quotient $BWM_{f}(\pm q^{k},q)/J_{f}(\pm q^{k},q)$ is semisimple
(see \cite[Cor. 5.6]{w2} for details). 

In examining the Kauffman link polynomials, 
we will draw on the structures of the semisimple quotients
$BWM_{f}(-q^{2n},q)/J_{f}(-q^{2n},q)$ and $BWM_{f}(q^{2n},-q)/J_{f}(q^{2n},-q)$, which are 
encoded in the relevant truncated Bratelli diagrams. 
These truncated Bratteli diagrams are identical, which arises directly from the fact that
$Q_{\lambda}(-q^{2n},q) = Q_{\lambda}(q^{2n},-q)$ (Lemma \ref{lem:JCBsong}) and
may be related to the fact that $BWM_{f}(-q^{2n},q)$ and $BWM_{f}(q^{2n},-q)$ are isomorphic 
algebras \cite{w2}. 

We now describe how to construct the truncated Bratelli diagram for $BWM_{f}(-q^{2n},q)/J_{f}(-q^{2n},q)$ \cite{w2}:
we similarly obtain the truncated Bratteli diagram for $BWM_{f}(q^{2n},-q)/J_{f}(q^{2n},-q)$.

Let $q$ be generic. Initially, we inductively obtain a subgraph $Y(-q^{2n},q)$ of the Young lattice
as follows:
\begin{itemize}
\item[(i)]  Firstly fix the Young diagram with no boxes to be a vertex in $Y(-q^{2n},q)$.
\item[(ii)] The inductive step is: assume that the Young diagram $\lambda$ is a vertex in
$Y(-q^{2n},q)$ and that the Young diagram $\mu$ differs from $\lambda$ by exactly one box.
Then $\mu$ is also a vertex in $Y(-q^{2n},q)$ if the function $Q_{\mu}(-q^{2n},q)$ 
given in Eq. (\ref{eq:themagixQpolynomial}) is non-zero. 
\end{itemize}
Explicitly from \cite{blumen05,w2}, a Young diagram $\lambda$ is a vertex in $Y(-q^{2n},q)$ if and only if $\lambda_{1}' + \lambda_{2}' \leq {2n+1}$ where $\lambda_{i}'$ is the number of boxes in the 
$i^{th}$ column of $\lambda$ from the left.

The truncated Bratteli diagram for $BWM_{f}(-q^{2n},q)/J_{f}(-q^{2n},q)$ is then the subgraph of 
the Bratteli diagram for $BWM_{f}$ obtained by removing all vertices that do not belong to 
$Y(-q^{2n},q)$. We show the truncated Bratteli diagram for $BWM_{f}(-q^{2},q)/J_{f}(-q^{2},q)$
up to the $4^{th}$ level in Figure 4.

The truncated Bratteli diagram for $BWM_{f}(q^{2n},-q)/J_{f}(q^{2n},-q)$ is obtained in the same
way as is the truncated Bratteli diagram for $BWM_{f}(-q^{2n},q)/J_{f}(-q^{2n},q)$ except that
we replace $(-q^{2n},q)$ with $(q^{2n},-q)$ throughout.
The fact that $Q_{\lambda}(-q^{2n},q) = Q_{\lambda}(q^{2n},-q)$ 
means that the truncated Bratteli diagrams for
$BWM_{f}(-q^{2n},q)/J_{f}(-q^{2n},q)$ and $BWM_{f}(q^{2n},-q)/J_{f}(q^{2n},-q)$
are identical.

\begin{figure}[hbt]
\label{fig:brattelidiagramforquotientalgebra}
\begin{center}
\setlength{\unitlength}{0.14mm}
\begin{picture}(460,440) \large\sf
\thinlines

\put(190,430){\circle{20}} 

\put(190,395){\line(0,1){20}}

\put(180,370){\framebox(20,20){}}

\put(120,300){\framebox(20,20){}}
\put(120,280){\framebox(20,20){}}

\put(170,290){\framebox(20,20){}}
\put(190,290){\framebox(20,20){}}

\put(250,300){\circle{20}}

\put(80,160){\framebox(20,20){}}
\put(80,180){\framebox(20,20){}}
\put(80,200){\framebox(20,20){}}
\put(100,225){\line(1,2){25}} 
\put(130,275){\line(1,-4){15}}

\put(130,170){\framebox(20,20){}}
\put(130,190){\framebox(20,20){}}
\put(150,190){\framebox(20,20){}}
\put(152,215){\line(1,2){35}} 
\put(192,285){\line(1,-2){40}}

\put(200,180){\framebox(20,20){}}
\put(220,180){\framebox(20,20){}}
\put(240,180){\framebox(20,20){}}

\put(290,180){\framebox(20,20){}}
\put(210,285){\line(1,-1){80}} 
\put(255,285){\line(1,-2){40}} 
\put(145,275){\line(2,-1){140}}

\put(135,325){\line(1,1){40}} 
\put(190,315){\line(0,1){50}} 
\put(205,365){\line(3,-4){38}} 

\put(160,60){\framebox(20,20){}}
\put(160,40){\framebox(20,20){}}
\put(180,60){\framebox(20,20){}}
\put(200,60){\framebox(20,20){}}
\put(150,163){\line(1,-3){26}} 
\put(190,85){\line(1,3){30}} 
\put(105,160){\line(3,-1){230}}

\put(240,60){\framebox(20,20){}}
\put(260,60){\framebox(20,20){}}
\put(280,60){\framebox(20,20){}}
\put(300,60){\framebox(20,20){}}
\put(225,175){\line(1,-2){45}} 
\put(95,155){\line(5,-2){175}} 

\put(340,60){\framebox(20,20){}}
\put(340,40){\framebox(20,20){}}
\put(155,165){\line(5,-2){195}} 
\put(310,175){\line(1,-1){90}} 
\put(310,175){\line(1,-2){45}} 
\put(310,175){\line(3,-2){130}} 

\put(380,60){\framebox(20,20){}}
\put(400,60){\framebox(20,20){}}
\put(165,165){\line(3,-1){240}} 
\put(225,175){\line(2,-1){180}} 

\put(450,70){\circle{20}} 

\end{picture}
\caption{The Bratteli diagram for $BWM_{f}(-q^{2},q)/J_{f}(-q^{2},q)$ up to the $4^{th}$ level inclusive}  
\end{center}
\end{figure}
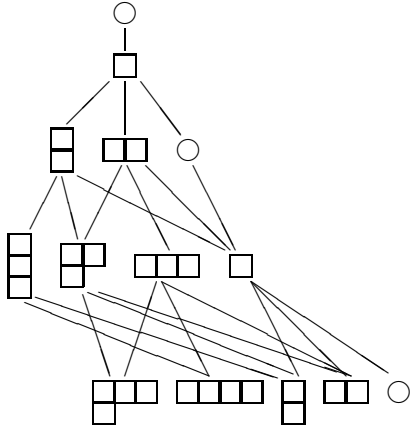

Note that
$BWM_{f}(-q^{2n},q)/J_{f}(-q^{2n},q) \cong \bigoplus_{\lambda} M_{b(\lambda)}(\mathbb{C})$
where the direct sum is over all vertices $\lambda$ on the $f^{th}$ level of the truncated
Bratteli diagram for $BWM_{f}(-q^{2n},q)/J_{f}(-q^{2n},q)$
and $b(\lambda)$ is the number of paths of length $f$ 
in the truncated Bratteli diagram ending at $\lambda$.

We are almost at the point where we can write down a basis for 
$BWM_{f}(-q^{2n},q)/J_{f}(-q^{2n},q)$.
Fix $\Omega_{f}(-q^{2n},q)$ to be the set of pairs $(R,S)$ 
of paths of length $f$ in the truncated Bratteli diagram for
$BWM_{f}(-q^{2n},q)/J_{f}(-q^{2n},q)$ 
where $shp(R)=shp(S)$.
The matrix units 
$$\left\{ E_{RS}(-q^{2n},q) \in BWM_{f}(-q^{2n},q) | \ (R,S) \in \Omega_{f}(-q^{2n},q) \right\},$$
obtained by replacing the indeterminates $r$ and $s$ with the complex numbers $-q^{2n}$ and $q$, 
respectively, in the relevant matrix units of $BWM_{f}$, are all well-defined and non-zero \cite{blumen05}.
It is very important to note that $\mathrm{tr}(E_{SS}(-q^{2n},q)) \neq 0$ 
for all $(S,S) \in \Omega_{f}(-q^{2n},q)$ and that 
$E_{RS}(-q^{2n},q) \notin J_{f}(-q^{2n},q)$ for all  $(R,S) \in \Omega_{f}(-q^{2n},q)$.

\section{Technical results}
\label{sec:technicalresults}

\subsection{Definition of $\mathrm{tr}$}

For each element $a \in BWM_{f+1}$, there exists a unique element
$\epsilon_{f}(a) \in BWM_{f}$ such that
$e_{f+1}ae_{f+1} = x \epsilon_{f}(a)e_{f+1}$ where
$\epsilon_{f}: BWM_{f+1} \rightarrow BWM_{f}$ is a linear map.
The $\mathrm{tr}$ functional on $BWM_{f}$ is then inductively defined by
$\mathrm{tr}(1) = 1$ and 
$\mathrm{tr}(a) = \mathrm{tr}(\epsilon_{f}(a))$ for $a \in BWM_{f+1}$.
In particular, $\mathrm{tr}(e_{i}) = 1/x$ and
$\mathrm{tr}(g_{i}^{\pm 1}) = r^{\pm 1}/x$ for all $i$.

\subsection{Definition of $Q_{\lambda}(r,s)$}

Let $\lambda$ be a Young diagram, let
$(i,j)$  denote the box in the $i^{th}$ row and the $j^{th}$ column of $\lambda$ 
and let $\lambda_{i}$ (resp. $\lambda_{j}'$) 
denote the number of boxes in the $i^{th}$ row (resp. $j^{th}$ column) of $\lambda$.
Denote the Young diagram $\lambda$ by 
$\lambda = [\lambda_{1}, \lambda_{2}, \ldots, \lambda_{k}]$ where the $i^{th}$ row of the Young diagram contains 
$\lambda_{i}$ boxes for each $i=1, 2, \ldots, k$, 
and the $l^{th}$ row contains no boxes for each $l > k$.
The function $Q_{\lambda}(r,s)$ is
\begin{eqnarray}
Q_{\lambda}(r,s) & = & 
\prod_{(j,j) \in \lambda} 
\frac{rs^{\lambda_{j} - \lambda_{j}'}-r^{-1}s^{-\lambda_{j} +\lambda_{j}'}+
s^{\lambda_{j} + \lambda_{j}'-2j+1}-s^{-\lambda_{j} -\lambda_{j}'+2j-1}}
{s^{h(j,j)}-s^{-h(j,j)}} \nonumber \\
& & \hspace{5mm} \times
\prod_{(i,j) \in \lambda, i \neq j} \frac{r s^{d(i,j)}-r^{-1}s^{-d(i,j)}}{s^{h(i,j)}-s^{-h(i,j)}},
\label{eq:themagixQpolynomial}
\end{eqnarray}
where the hooklength $h(i,j)$ is defined by
$h(i,j) = \lambda_{i} - i + \lambda_{j}' - j +1$, and where
$$d(i,j) = \left\{ \begin{array}{ll}
 \lambda_{i} + \lambda_{j}-i-j+1, & \hspace{5mm} \mbox{if } i \leq j, \\
-\lambda_{i}' - \lambda_{j}' + i + j-1, & \hspace{5mm} \mbox{if } i > j.
\end{array} \right.$$
Intuitively, the hooklength $h(i,j)$ is the number of boxes 
in the hook cornered on the box $(i,j)$, i.e. the number of boxes
below the $(i,j)$ box in the $j^{th}$ column 
plus the number of boxes to the right of the $(i,j)$ box in the $i^{th}$ row, plus one.

\subsection{Technical lemmas}

\begin{lemma}
\label{lem:powersofg}
For all integers $f, m \geq 2$, 
\begin{itemize}
\item[(i)]
$(g_{1})^{m} \in BWM_{f}$ can be written as
\begin{equation}
\label{eq:powersofgrelation}
(g_{1})^{m} = a_{m}(r,s) + b_{m}(r,s) g_{1} + c_{m}(r,s)e_{1},
\end{equation}
where $a_{m}(r,s), b_{m}(r,s), c_{m}(r,s) \in \mathbb{Z}[[r,r^{-1},s,s^{-1}]]$, and
\item[(ii)]
\begin{eqnarray*}
a_{m}(r,s) & \in & \left\{ \bigoplus_{i} \mathbb{Z} r^{\gamma_{i}}s^{{\delta}_{i}}
\mbox{ where } (\gamma_{i} + \delta_{i}) \bmod{2} \equiv (m) \bmod{2} \ \forall i
   \right\}, \\
b_{m}(r,s) & \in & \left\{ \bigoplus_{j} \mathbb{Z} r^{\gamma_{j}}s^{{\delta}_{j}}
\mbox{ where } (\gamma_{j} + \delta_{j}) \bmod{2} \equiv (m+1) \bmod{2} \ \forall j
   \right\}, \\
c_{m}(r,s) & \in & \left\{ \bigoplus_{k} \mathbb{Z} r^{\gamma_{k}}s^{{\delta}_{k}}
\mbox{ where } (\gamma_{k} + \delta_{k}) \bmod{2} \equiv (m) \bmod{2} \ \forall k
   \right\}. 
\end{eqnarray*}
\end{itemize}

\end{lemma}
\begin{proof}
\begin{itemize}
\item[(i)] 
We firstly note that
$$(g_{1})^{2} = 1 + (s-s^{-1})(g_{1}-r^{-1}e_{1}).$$
Assume now that Eq. (\ref{eq:powersofgrelation}) is true for some $m \geq 2$, then
\begin{eqnarray}
(g_{1})^{m+1} 
   & = & a_{m}(r,s) g_{1} + b_{m}(r,s)(g_{1})^{2} + c_{m}(r,s)e_{1}g_{1} \nonumber  \\
   & = & b_{m}(r,s) + \left(a_{m}(r,s) + b_{m}(r,s)(s-s^{-1})\right)g_{1} \nonumber \\
   &   &      + \left(-b_{m}(r,s)r^{-1}(s-s^{-1}) + c_{m}(r,s)r^{-1}\right)e_{1},
                \label{eq:powersofginductionmplus1}
\end{eqnarray}
proving (i).
\item[(ii)] The result is true for $m=2$ by inspection and follows for all $m \geq 2$
by induction.
\end{itemize}
\end{proof}

\begin{lemma}
\label{lem:JCBsong}
For each Young diagram $\lambda$, 
\begin{equation}
\label{eq:hideandseek}
Q_{\lambda}(-q^{2n},q) = Q_{\lambda}(q^{2n},-q).
\end{equation}
\end{lemma}
\begin{proof}
Simple calculations show that (\ref{eq:hideandseek}) is true if and only if
\begin{eqnarray*}
 \prod_{\stackrel{(i,j) \in \lambda}{i \neq j}} -q^{2n+d(i,j)}+q^{-2n-d(i,j)}
& = & 
\left(\prod_{\stackrel{(i,j) \in \lambda}{i < j}} (-1)^{\lambda_{j}'+\lambda_{j}}(q^{2n+d(i,j)}-q^{-2n-d(i,j)})\right) \\ 
& & \times \left(\prod_{\stackrel{(i,j) \in \lambda}{i > j}} (-1)^{\lambda_{i}+\lambda_{i}'}(q^{2n+d(i,j)}-q^{-2n-d(i,j)})\right)
\end{eqnarray*}
and this last equation is true if
\begin{equation}
\label{eq:makethingsright}
\prod_{\stackrel{(i,j) \in \lambda}{i \neq j}}(-1) = 
\left( \prod_{\stackrel{(i,j) \in \lambda}{i<j}} (-1)^{\lambda_{j}' + \lambda_{j} }  	\right) 
\left( \prod_{\stackrel{(i,j) \in \lambda}{i>j}} (-1)^{\lambda_{i} + \lambda_{i}'}   \right).
\end{equation}
We now show that (\ref{eq:makethingsright}) is true.
Define the following sets:
\begin{eqnarray*}
\mathrm{Hor}_{k} & = & \left\{ (k,j) \in \lambda | \ j=1, 2, \ldots, \min{ \{k-1, \lambda_{k}   \} } \right\} \\
\mathrm{Ver}_{k} & = & \left\{ (i,k) \in \lambda | \ i=1, 2, \ldots, \min{ \{ k-1, \lambda_{k}' \} } \right\}.
\end{eqnarray*}
Noting that $|\mathrm{Ver}_{k} \cap \mathrm{Hor}_{l}| = 0$ for all $k$ and $l$ and that
$|\mathrm{Hor}_{k} \cap \mathrm{Hor}_{l}| = 0 = |\mathrm{Ver}_{k} \cap \mathrm{Ver}_{i}|$ for all $k \neq i$,  
it follows that (\ref{eq:makethingsright}) is true if the following equation holds for each $k$: 
\begin{equation}
\label{eq:politesociety}
(-1)^{|\mathrm{Hor}_{k} \cup \mathrm{Ver}_{k}|} = 
\left( \prod_{(i,k) \in \mathrm{Ver}_{k}} (-1)^{\lambda_{k}' + \lambda_{k} }  	\right) 
\left( \prod_{(k,j) \in \mathrm{Hor}_{k}} (-1)^{\lambda_{k} + \lambda_{k}'}   \right).
\end{equation}
If $|\mathrm{Hor}_{k} \cup \mathrm{Ver}_{k}|$ is even, 
the right hand side of (\ref{eq:politesociety}) clearly equals $1$
as $\mathrm{Hor}_{k}$ and $\mathrm{Ver}_{k}$ are disjoint.
Alternatively, if $|\mathrm{Hor}_{k} \cup \mathrm{Ver}_{k}|$ is odd, 
then $\lambda_{k} \leq k-2$ and/or $\lambda_{k}' \leq k-2$.
If $\lambda_{k}' \leq k-2$, then $\lambda_{k} \leq k-1$ as $\lambda$ is a Young diagram.
Similarly, if $\lambda_{k} \leq k-2$, then $\lambda_{k}' \leq k-1$ as $\lambda$ is a Young diagram.
In both cases it follows that $\lambda_{k}' = |\mathrm{Ver}_{k}|$ and $\lambda_{k} = |\mathrm{Hor}_{k}|$.
If $|\mathrm{Hor}_{k} \cup \mathrm{Ver}_{k}|$ is odd, $\lambda_{k} + \lambda_{k}'$ is also odd as
$\lambda_{k} + \lambda_{k}' = |\mathrm{Hor}_{k}| + |\mathrm{Ver}_{k}| = 
|\mathrm{Hor}_{k} \cup \mathrm{Ver}_{k}|$, and clearly
the right hand side of (\ref{eq:politesociety}) equals $-1$.
Thus (\ref{eq:politesociety}) is true for each $k$, 
from which it follows that (\ref{eq:makethingsright}) is true,
which completes the proof of the lemma.
\end{proof}

\section*{Acknowledgments}
I would like to thank Prof. H. Wenzl for his comments on an 
early draft of this paper, Dr. N. Geer for discussions
during which the ideas in this paper crystallised and valuable comments from the referee.

\end{document}